\documentclass[12pt,letterpaper]{amsart}
\usepackage{geometry}
\usepackage{amssymb}
\usepackage{latexsym,url,setspace}
\usepackage{mathpazo} 
\usepackage{hyperref}
\newtheorem{theorem}{Theorem}
\newtheorem{lemma}{Lemma}
\newtheorem{corollary}{Corollary}
\newcommand{\C}{\mathbb{C}}

\newcommand{\zb}{\overline{z}}

\newcommand{\D}{\Omega}

\newcommand{\Dc}{\overline{\Omega}}
\newcommand{\dbar}{\overline{\partial}}
\newcommand{\ep}{\varepsilon}
\newcommand{\ob}{\overline{\omega}}
\newcommand{\sumprime}{\sideset{}{'}\sum}

\onehalfspace

\title{Convex domains, Hankel operators, and maximal estimates}

\author{Mehmet \c{C}el\.ik}
\address[Mehmet \c{C}elik]{Texas A\&M University-Commerce, 
	Department of Mathematics, Commerce, TX 75429, USA}
\email{mehmet.celik@tamuc.edu}

\author{S\"{o}nmez \c{S}ahuto\u{g}lu}
\address[S\"{o}nmez \c{S}ahuto\u{g}lu]{ University of Toledo, 
	Department of Mathematics \& Statistics, Toledo, OH 43606, USA}
\email{sonmez.sahutoglu@utoledo.edu}

\author{Emil J. Straube}
\address[Emil J. Straube]{Texas A\&M University, Department of Mathematics, 
	College Station, TX, 77843, USA}
\email{straube@math.tamu.edu}

\date{February 24, 2019}
\begin{document}

\thanks{Supported in part by Qatar National Research Fund Grant 
	NPRP 7-511-1-98 and by the Erwin Schr\"{o}dinger International 
	Institute for Mathematics and Physics, workshop 
	\emph{Analysis and CR Geometry}, Dec. 2018.}

\begin{abstract}
Let $1\leq q\leq (n-1)$. We first show that a necessary condition for 
a Hankel operator on 
$(0,q-1)$-forms on a convex domain to be compact is that its symbol 
is holomorphic along $q$-dimensional analytic varieties in the boundary. 
Because maximal estimates (equivalently, a comparable eigenvalues 
condition on the Levi form of the boundary) turn out to be favorable 
for compactness of Hankel operators, this result then implies that 
on a convex domain, maximal estimates exclude analytic varieties 
from the boundary, except ones of top dimension $(n-1)$  
(and their subvarieties). Some of our techniques apply to general 
pseudoconvex domains to show that if the Levi form has comparable 
eigenvalues, or equivalently, if the domain admits maximal estimates, 
then compactness and subellipticity hold for forms at \emph{some} 
level $q$ if and only if they hold at \emph{all} levels. 
\end{abstract}

\maketitle

\section{Introduction and Results}

Let $\Omega$ be a smooth ($C^{\infty}$) bounded pseudoconvex domain 
in $\mathbb{C}^{n}$. Let $U$ be a  neighborhood of a point 
$z\in b\Omega$ small enough that there exist smooth vector fields 
$L_{1}, L_{2}, \ldots, L_{n}$ 
in $U$, of type $(1,0)$, pointwise orthonormal, so that $L_{1}, \ldots,L_{n-1}$ 
are tangential to $b\Omega$ at the boundary, and $L_{n}$ is the (complex) 
normal. We use the customary notation $\omega_{1},\omega_{2},\ldots,\omega_{n}$ 
to denote the $(1,0)$-forms dual to $L_{1},\ldots,L_{n}$.  For a boundary point 
$z\in b\Omega \cap U$, $L_{1}(z), \ldots, L_{n-1}(z)$ form an orthonormal basis 
for $T^{1,0}_{z}(b\Omega)$. The collection 
$(U, L_{1},\cdots,L_{n}, \omega_{1},\cdots, \omega_{n})$
is usually referred to as a special boundary chart and/or frame.
The Levi form of the boundary is defined via 
$[L,\overline{L}] = \mathcal{L}(L,\overline{L})T \mod T^{1,0}\oplus T^{0,1}$, 
where $T$ is the familiar `bad' direction inside the tangent space, 
purely imaginary, normalized and chosen so that $\mathcal{L}$ 
is positive semi-definite. 

Denote by $\lambda_{j}(z)$, $1\leq j\leq n-1$, the eigenvalues of the Levi form 
of $b\Omega$ at the point $z\in b\Omega$, in increasing order. Strictly speaking, 
we mean the eigenvalues of the matrix that represents the Levi form with respect 
to a basis $L_{1}, \ldots,L_{n-1}$ as above; as long as we insist on orthonormal bases, 
these eigenvalues do not depend on the basis chosen
\footnote{Even if we were to drop the requirement that the basis be orthonormal, 
the next condition would remain independent of the basis chosen, in view of a 
theorem of Ostrowski which relates eigenvalues of matrices of the form $M$ 
and $\overline{S}^{T}MS$; see for example \cite{HJ85}, Theorem 4.5.9.}. 
We say that \emph{the Levi form of $b\Omega$ satisfies a comparable 
eigenvalues condition at level $q$ in $U\cap b\Omega$}, if there exists a 
constant $C>0$ such that 
$C( \lambda_{1}(z)+ \cdots +\lambda_{n-1}(z)) 
\leq  \sum_{s=1}^{q}\lambda_{j_{s}}(z) 
\leq \lambda_{1}(z)+ \cdots +\lambda_{n-1}(z)$ 
for any $q$-tuple $(j_{1}, \ldots,j_q)$ and $z\in b\Omega$. That is, the sum 
of any $q$ eigenvalues is comparable to the trace. Note that the second 
inequality is trivially satisfied because $\Omega$ is pseudoconvex. This 
condition is easily seen to be equivalent to sums of $q$ eigenvalues being 
comparable.\footnote{Note that $\lambda_{1}+\cdots+\lambda_{n-1} = 
\left(\begin{array}{c}n-2\\ 
 q-1\end{array}\right)^{-1}\sideset{}{'}\sum_{|J|=q}(\lambda_{j_{1}}
+\cdots+\lambda_{j_{q}})$,
 where the summation is over strictly increasing multi-indices $J$. 
 Thus if the $q$-sums 
compare, the trace also compares to any $q$-sum. A similar observation for 
$(\lambda_{1}+\cdots+\lambda_{q+1})$ shows that if the comparable 
eigenvalues condition holds at level $q$, it also holds at level $(q+1)$.} 
We say that \emph{the Levi form of $b\Omega$ satisfies a 
comparable eigenvalues condition at level $q$} if every point $z\in b\Omega$ 
has a neighborhood $U$ so that the condition is satisfied in $U\cap b\Omega$. 
Because $b\Omega$ is compact, we may take the constant $C$ to be independent 
of $z\in b\Omega$.

The comparable eigenvalues condition is important because it is equivalent to 
an $L^{2}$ estimate in the $\overline{\partial}$-Neumann problem that is better 
than the usual estimate on a pseudoconvex domain. Namely, if $b\Omega$ 
satisfies a comparable eigenvalues condition at level $q$, we have the estimate
\begin{equation}\label{EqnMaximalEst}
\sum_{j=1}^{n}\|\overline{L_{j}}f\|^{2} + \sum_{j=1}^{n-1}\| L_{j}f\|^{2} 
\lesssim(\|\dbar f\|^{2}+\|\dbar^*f\|^{2}+\|f\|^{2}) 
\end{equation}
for any 
$f\in Dom(\dbar)\cap Dom(\dbar^*)\cap C^{\infty}_{(0,q)}(\overline{\Omega})$ 
that is supported in a special boundary chart. Here and throughout the paper, we 
employ the customary convention that $\lesssim$ indicates an estimate with a constant 
that is independent of all relevant quantities.
For proofs, see \cite{Derridj78}, Th\'{e}or\`{e}me 3.1 for $q=1$, 
and \cite{BenMoussa00}, Th\'{e}or\`{e}me 3.7 for $q>1$. (Since we work 
on pseudoconvex domains, the term $\|f\|^{2}$ on the right hand side is 
dominated by the others, and so will not be needed.) The first term 
in \eqref{EqnMaximalEst} is always dominated by the right hand side, in view of 
the Morrey-Kohn-H\"{o}rmander inequality, however, the second term is not in 
general. So the point of \eqref{EqnMaximalEst} is that all, not just the barred, 
complex tangential derivatives of $f$ are controlled by 
$\|\overline{\partial}f\|+\|\overline{\partial}^{*}f\|$. Such estimates are referred 
to as \emph{maximal estimates}. We refer the reader to the introduction 
of \cite{Koenig15} for an account of the genesis of this terminology, and the 
important role such estimates play in the theory of the 
$\overline{\partial}$-Neumann problem.
\smallskip 

Next we define the Hankel operators on $(0,q)$-forms for $0\leq q\leq n$ as 
follows. Let $K^2_{(0,q)}(\D)$ denote the set of square integrable $\dbar$-closed 
$(0,q)$-forms on $\D$ and $P_q:L^2_{(0,q)}(\D) \to K^2_{(0,q)}(\D)$ be the 
Bergman projection. The Hankel operator with symbol 
$\phi\in L^{\infty}(\D)$ is the operator 
$H^q_{\phi}:K^2_{(0,q)}(\D)\to L^2_{(0,q)}(\D)$, 
\begin{equation}\label{Hankel}
H^q_{\phi}f=\phi f-P_q(\phi f)
\end{equation}
for $f\in K^2_{(0,q)}(\D)$. When $\phi\in C^1(\Dc)$, 
Kohn's formula, $P_q=I-\dbar^*N_{q+1}\dbar$, implies that 
\begin{equation}\label{KohnHankel}
H^{q}_{\phi}f=\dbar^*N_{q+1}(\dbar \phi\wedge f)
\end{equation}
for  $f\in K^2_{(0,q)}(\D)$. Here, $N_{q+1}$ denotes the 
$\overline{\partial}$-Neumann operator on $(0,q+1)$-forms. 
In the following theorem, $A^2_{(0,q)}(\D)\subset K^{2}_{(0,q-1)}(\Omega)$ 
denotes the space of  $(0,q)$-forms with square integrable holomorphic  
coefficients, and $\mathbb{D}^{q}$ is the unit polydisc, i.e. the $q$-fold 
product of the unit disc in $\mathbb{C}^{q}$.

Our first result gives a necessary condition for compactness of Hankel 
operators; the case $q=1$ and $n=2$ is in \cite{ClosCelikSahutoglu}
(for symbols in $C(\overline{\Omega})$), 
the case $q=1$ but general $n$ is in \cite{CuckovicSahutoglu09}
(for symbols in $C^{\infty}(\overline{\Omega})$).

\begin{theorem}\label{ThmNonCompact}
Let $\D$ be a bounded convex domain in $\C^n$ for $n\geq 2$.   
Assume that $\phi\in C^1(\Dc)$ and there exists a holomorphic 
embedding $\psi:\mathbb{D}^q\to b\D$ for some $1\leq q\leq n-1$ 
such that $\phi\circ \psi$ is not holomorphic. Then $H^{q-1}_{\phi}$ 
is not compact on $A^2_{(0,q-1)}(\D)$ (and \emph{a fortiori} 
not on $K^{2}_{(0,q-1)}(\Omega)$).
\end{theorem}
That is, for $H^{q-1}_{\phi}$ to be compact (even on $A^{2}_{(0,q-1)}(\Omega)$) 
it is necessary that the symbol $\phi$ is holomorphic along (the regular part of) 
$q$-dimensional, and thus higher dimensional, varieties in the boundary. 
Since $\Omega$ is convex, such varieties are 
necessarily contained in affine varieties, see \cite{FuStraube98}, 
Theorem 1.1 and section 2, and \cite{CuckovicSahutoglu09}, Lemma 2.  
The proof of Theorem \ref{ThmNonCompact}, given in section 2, 
combines ideas from \cite{FuStraube98} and 
\cite{ClosCelikSahutoglu} in a fairly straightforward way.

In view of the 'if and only if' nature of the results in \cite{FuStraube98}, one 
might expect that the converse of Theorem \ref{ThmNonCompact} also holds. 
This is known for (convex) domains in $\mathbb{C}^{2}$ 
(\cite{CuckovicSahutoglu09}, Theorem 3), and we can verify it in many cases, 
but the general case (for convex domains in dimension $n\geq 3$) remains open.

\smallskip

Suppose we have an estimate whose right hand side depends only on 
$\|\overline{\partial}f\|$ and $\|\overline{\partial}^{*}f\|$ (possibly modulo 
`weak' terms), as in \eqref{EqnMaximalEst}, at the level of $(q+1)$-forms. 
In addition to \eqref{EqnMaximalEst}, examples include compactness 
estimate (\eqref{comp} below) and subelliptic estimate (\eqref{subelliptic} 
below). In order to derive an analogous estimate for $q$-forms, 
it is natural to take a $q$-form $f$ and produce $(q+1)$-forms 
$f^{k}:=f\wedge\overline{\omega_{k}}$, $k=1,\ldots,(n-1)$ (say $f$ is 
supported in a local frame), control the relevant norm of $f$ by 
those of the $f^{k}$, apply the known estimate to the $(q+1)$-forms 
$f^{k}$, and finally control $\|\overline{\partial}(f^{k})\|$ and 
$\|\overline{\partial}^{*}(f^{k})\|$ by $\|\overline{\partial}f\|$ and 
$\|\overline{\partial}^{*}f\|$. This is no problem for 
 $\overline{\partial}(f^{k})=\overline{\partial}f\wedge\overline{\omega_{k}} 
 + (-1)^{q}f\wedge\overline{\partial}(\overline{\omega_{k}})$; it is controlled 
 by $\|\overline{\partial}f\|+\|f\|$, hence by 
 $\|\overline{\partial}f\|+\|\overline{\partial}^{*}f\|$. The form 
 $\overline{\partial}^{*}(f^{k})$ takes more care. First, if $f$ is smooth 
 enough (say in $C^{\infty}_{(0,q)}(\overline{\Omega})$ for simplicity), 
 then $f^{k}=f\wedge\overline{\omega_{k}}\;$ is in 
 $dom(\overline{\partial}^{*})$ if $f$ is.
 Indeed, since the normal components of both $f$ and $\overline{\omega_{k}}$ 
 vanish on the boundary, so does that of $f\wedge\overline{\omega_{k}}$. 
 Computation of $\overline{\partial}^{*}(f^{k})$ then reveals that to 
 control $\|\overline{\partial}^{*}(f^{k})\|$, one needs not only 
 $\|\overline{\partial}^{*}f\|$ and $\|f\|$, but also $\|L_{k}f\|$ 
(see \eqref{Eqn11} in section \ref{SectionAppendix} below). 
So in order for the above scheme to work, the latter term needs to be controlled 
by $\|\overline{\partial}f\|+\|\overline{\partial}^{*}f\|$. This, however, is precisely 
what  the condition of maximal estimates ensures. Theorems \ref{ThmHankel} 
and \ref{ThmComp} below take advantage of this observation.

\begin{theorem}\label{ThmHankel}
Let $\D$ be a smooth bounded convex domain in $\C^n$, $n\geq 2$ and 
$1\leq q\leq n-1$. Assume that the Levi form of $b\Omega$ satisfies a 
comparable eigenvalues condition at level $q$. Let $\phi \in C^1(\Dc)$ 
such that $\phi\circ \psi$ is holomorphic for every holomorphic 
embedding $\psi:\mathbb{D}^{n-1}\to b\D$. Then the Hankel operator 
$H^{q-1}_{\phi}:K^2_{(0,q-1)}(\D)\to L^2_{(0,q-1)}(\D) $ is compact. 
\end{theorem}

Note that the symbol $\phi$ is assumed holomorphic only on $(n-1)$-dimensional 
varieties, while the Hankel operator is on $(0,q-1)$-forms. Combined with 
Theorem \ref{ThmNonCompact}, this `discrepancy' leads to the following corollary. 
Its gist is that on convex domains, varieties in the boundary, apart from the ones 
in top dimension, are obstructions to maximal estimates (equivalently, to 
comparable eigenvalues conditions). 

We call the image of an embedding $\psi$ as in Theorem \ref{ThmHankel} an $(n-1)$ dimensional analytic polydisc.
\begin{corollary}\label{Cor1}
Let $\D$ be a smooth bounded convex 
domain in $\C^n$, $n\geq 2$ and $1\leq q\leq n-1$. Denote by
$A$ the union of the $(n-1)$-dimensional analytic polydiscs in $b\D$.
Assume that $\D$ satisfies maximal estimates for $(0,q)$-forms. Then 
$b\D\setminus \overline{A}$ contains no
$q$-dimensional analytic varieties. 
\end{corollary}

Convex domains in $\mathbb{C}^{2}$ (where maximal estimates hold 
trivially on $(0,1)$-forms) show that the requirement in the corollary 
that the varieties be outside $\overline{A}$ cannot be dropped. We also note that
$(n-1)$--dimensional polydiscs in $b\Omega$ are open subsets of a complex hyperplane.
Indeed, the argument in \cite{FuStraube98}, section 2 (see also \cite{CuckovicSahutoglu09}, 
Lemma 2) shows that if the supporting real hyperplane to $\Omega$ at a point is $\{x_{n}=0\}$
in suitable coordinates, then the embedding $\psi$ maps into the complex hyperplane $\{z_{n}=0\}$.

\smallskip

A portion of the technique in the proof of Theorem \ref{ThmHankel} 
leads to an interesting percolation phenomenon for compactness and 
subellipticity in the $\overline{\partial}$-Neumann problem on domains 
with maximal estimates. We first recall these notions.

\emph{The $\overline{\partial}$-Neumann problem is said to satisfy a
compactness estimate} for $(0,q)$-forms if the following holds: 
for every $\varepsilon >0$, there exists a 
constant $C_{\varepsilon}$ such that
\begin{equation}\label{comp}
\|f\|^{2} \leq \varepsilon\left(\|\overline{\partial}f\|^{2} 
+ \|\overline{\partial}^{*}f\|^{2}\right) 
+ C_{\varepsilon}\|f\|_{-1}^{2} \,,\, f\in dom(\overline{\partial}) 
\cap dom(\overline{\partial}^{*}) \subset L^{2}_{(0,q)}(\Omega)\,.
\end{equation}
Here, $\|f\|_{-1}$ denotes the Sobolev-$(-1)$ norm. 
The $\overline{\partial}$-Neumann problem is said to be subelliptic for 
$(0,q)$-forms if there exists $\varepsilon >0$ and a constant $C$ such that
\begin{equation}\label{subelliptic}
 \|f\|_{\varepsilon}^{2} \leq C\left(\|\overline{\partial}f\|^{2} 
 + \|\overline{\partial}^{*}f\|^{2}\right) \,,\, f\in dom(\overline{\partial})
 \cap dom(\overline{\partial}^{*})\subset L^{2}_{(0,q)}(\Omega) \,.
\end{equation}
Again, the subscript $\varepsilon$ denotes the Sobolev-$\varepsilon$ norm. 
We say that \emph{the $\overline{\partial}$-Neumann problem is subelliptic 
of order $\varepsilon$}. The relevance of estimates \eqref{comp} and 
\eqref{subelliptic} stems from their equivalence to compactness and 
subellipticity, respectively, of the $\overline{\partial}$-Neumann operator 
$N_{q}$ (\cite{StraubeBook}, \cite{KohnNirenberg65}, \cite{D'AngeloKohn99}). 

Compactness and subellipticity in the $\overline{\partial}$-Neumann problem
are known to percolate up the $\overline{\partial}$ complex (\cite{StraubeBook}, 
Proposition 4.5 and the remark following its proof, \cite{McNeal05}); the point of 
Theorem \ref{ThmComp} is that they percolate \emph{down} to level $q$ 
when the Levi form of $b\Omega$ satisfies a comparable eigenvalues 
condition at level $q$ (equivalently: when there are maximal estimates 
for $(0,q)$-forms). 

\begin{theorem}\label{ThmComp}
Let $\D$ be a smooth bounded pseudoconvex  domain in $\C^n$ for  
$n\geq 2$. Assume  that the Levi form of $b\Omega$ 
satisfies a comparable eigenvalues condition at level $q$ for some $q$,
$1\leq q\leq n-1$. Then  

(i) The $\overline{\partial}$-Neumann problem satisfies compactness 
estimate for $(0,q)$-forms if and only if it satisfies such estimate 
for $(0,n-1)$-forms. 

(ii) The $\overline{\partial}$-Neumann problem is subelliptic of order 
$\varepsilon$ for $(0,q)$-forms if and only if it is subelliptic of order 
$\varepsilon$ for $(0,n-1)$-forms, $0<\varepsilon\leq 1/2$.
\end{theorem}
The following corollary for domains with comparable eigenvalues of the 
Levi form, that is, domains which satisfy the comparable eigenvalues 
condition at level $q=1$, is immediate, but we formulate it for emphasis: 
if compactness or subelliptic estimates hold at \emph{some} form level, 
corresponding estimates hold at \emph{all} levels.

\begin{corollary}\label{CorComp}
 Let $\D$ be a smooth bounded pseudoconvex  
domain in $\C^n$, $n\geq 2$. Assume  that the eigenvalues of the Levi form 
of $b\Omega$ are comparable (equivalently, $\Omega$ admits maximal 
estimates for $(0,1)$--forms). Then

i) The $\overline{\partial}$-Neumann problem satisfies compactness estimate  
for $(0,q_{0})$-forms for some $q_{0}$, $1\leq q_{0}\leq (n-1)$, if and only if it 
satisfies compactness estimate for $(0,q)$-forms for all $q$, $1\leq q\leq (n-1)$.

ii) The $\overline{\partial}$-Neumann problem is subelliptic on $(0,q_{0})$-forms 
for some $q_{0}$, $1\leq q_{0}\leq (n-1)$, if and only if it is subelliptic on 
$(0,q)$-forms for all $q$, $1\leq q\leq (n-1)$.
\end{corollary}

We remark that the $\overline{\partial}$-Neumann problem is always 
subelliptic (and hence also compact) on $(0,n)$-forms.

\smallskip

The rest of the paper is organized as follows. In section 2 we prove 
Theorem \ref{ThmNonCompact}. Theorem \ref{ThmHankel} and 
Corollary \ref{Cor1} are shown in section 3. Section 4 contains the 
proof of Theorem \ref{ThmComp}. In the appendix, section 5, 
we compute $\overline{\partial}^{*}(f\wedge\overline{\omega_{k}})$ 
for $f \in dom(\overline{\partial}^{*})$.

\section{Proof of Theorem \ref{ThmNonCompact}}

\begin{proof}[Proof of Theorem \ref{ThmNonCompact}:]
The proof combines ideas from \cite{FuStraube98} and 
\cite{CuckovicSahutoglu09, ClosCelikSahutoglu}. (In turn, these ideas 
can be traced back at least to \cite{Catlin81, DP81}.) In particular, we 
follow the geometric setup in the proof of the implication (1) 
$\Rightarrow$ (2) in Theorem 1.1 in \cite{FuStraube98}. If $b\Omega$ 
contains a complex variety $V$ of dimension $q$ as in Theorem \ref{ThmNonCompact}, 
its convex hull $W$ is an affine variety in $b\Omega$ (\cite{CuckovicSahutoglu09}, 
Lemma 2, see also \cite{FuStraube98}, section 2) of dimension at least $q$. 
Because $\phi$ is not holomorphic on $V$, it is not holomorphic on $W$.
Consequently, there is a $q$-dimensional affine variety in $W\subset
b\Omega$  on which $\phi$ is not holomorphic. 
After a suitable affine change of 
coordinates, we may assume that $(2\mathbb{D})^{q}\times \{0\}
= \{(z',0)\in\mathbb{C}^{n}; z'\in(2\mathbb{D})^{q}\} \subset b\Omega$, 
where $\mathbb{D}$ is the unit disc in $\mathbb{C}$ and 
$z' = (z_{1}, \ldots, z_{q})$, and that $\partial \phi/\partial \zb_1(z)\neq 0$ 
when $|z_1|< 1$. Let $z''=(z_{q+1},\ldots, z_n)$. We set
$\D_1:=\{z''\in \C^{n-q}:(0,z'')\in \D\}$, and 
$\Omega_{2}:= \{z''\in \mathbb{C}^{n-q}: 2z''\in \Omega_{1}\}$. 
Convexity of $\Omega$ implies that 
$\mathbb{D}^{q} \times \Omega_{2} \subseteq \Omega$ (\cite{FuStraube98}, 
page 636): every point in this set is the midpoint of a line segment joining 
a point in $\mathbb{D}^{q} \times \{0\}$ to a point in $\{0\} \times \Omega_{1}$.

The crucial analytic fact from \cite{FuStraube98} is the following. There exists 
a bounded sequence $\{F_{j}\}_{j=1}^{\infty} \subset A^{2}(\Omega)$ such that 
the sequence $\{f_{j}\}_{j=1}^{\infty}$ of restrictions to $\Omega_{2}$, given 
by $f_{j}(z''):= F_{j}(0, z'')$, belongs to $A^{2}(\Omega_{2})$, but does not 
admit a convergent subsequence.
\footnote{The crux of the matter is that the restriction operator from 
	$A^{2}(\Omega_{1})$ to $A^{2}(\Omega_{2})$ is not compact; the proof 
	involves estimates on the Bergman kernel of $\Omega_{1}$. 
	The Ohsawa-Takegoshi extension theorem then allows to pass from a 
	sequence on $\Omega_{1}$ with the required property to a suitable 
	sequence on $\Omega$.}

For the rest of the argument, we follow 
\cite{CuckovicSahutoglu09,ClosCelikSahutoglu}, with appropriate 
modifications. Choose a radially symmetric non-negative function 
$\chi\in C^{\infty}_0(\mathbb{D})$  such that $\chi(\xi)=1$ for $|\xi|\leq 1/2$ 
and $\chi(\xi)=0$ for $|\xi|\geq 3/4$. We denote  
$\int_{\{|\xi|\leq 3/4\}}\chi(\xi)dV(\xi)=c_{\chi}>0$. Because 
 $\partial\phi/\partial\overline{z_{1}} \neq 0$ when $|z_{1}|<1$, we can 
 define $\gamma\in C(\overline{\Omega})$ via  the formula
\begin{equation}\label{gamma}
\gamma(z',z'')\frac{\partial \phi(z',z'')}{\partial \zb_1}= \chi(z_1)\cdots \chi(z_q)\; .
\end{equation}
Note that for $z''$ fixed, $\gamma(\cdot,z'')$ is compactly supported 
in $\mathbb{D}^{q}$, uniformly in $z''$. We will eventually have to 
approximate $\gamma$ by a smooth function, so let 
$\gamma_{1}\in C^{\infty}(\overline{\Omega})$ such that for 
$z''\in \Omega_{2}$, $\gamma_{1}(\cdot,z'')$ is compactly supported 
in $\mathbb{D}^{q}$. Let $F\in A^2(\D)$  and  
$\alpha=Fd\zb_2\wedge \cdots\wedge d\zb_q\in A^2_{(0,q-1)}(\D)$\footnote{When $q=1$, this definition is to be interpreted as $\alpha = F$.}. 
Denote by $\langle\;, \; \rangle$ the standard 
pointwise inner product on forms in $\C^q$.
Then, for $z''\in \Omega_{2}$, the mean value property for holomorphic 
functions gives
\begin{multline}\label{keycalc}
(c_{\chi})^qF(0,z'')=\int_{\mathbb{D}^q} \chi(z_1)\cdots \chi(z_q)F(z',z'')dV(z')\\
=\int_{\mathbb{D}^q} \gamma(z',z'')\frac{\partial \phi(z',z'')}{\partial \zb_1}F(z',z'')dV(z')
= \int_{\mathbb{D}^q} \langle \dbar \phi\wedge \alpha, \overline{\gamma}d\zb_1\wedge\cdots\wedge d\zb_q\rangle dV(z')\\
= \int_{\mathbb{D}^q} \langle \dbar \phi\wedge \alpha, \overline{\gamma_{1}}d\zb_1\wedge\cdots\wedge d\zb_q\rangle dV(z') + \int_{\mathbb{D}^q} \langle \dbar \phi\wedge \alpha, (\overline{\gamma}-\overline{\gamma_{1}})d\zb_1\wedge\cdots\wedge d\zb_q\rangle dV(z') \;.
\end{multline}
Because $F$ is holomorphic, $\overline{\partial}\alpha = 0$ implies
\begin{equation}\label{key}
\overline{\partial}\phi\wedge\alpha = \overline{\partial}(\phi\alpha) = \overline{\partial}(\phi\alpha - P_{q-1}(\phi\alpha)) =  \overline{\partial}H^{q-1}_{\phi}\alpha \; .
\end{equation}
Denote by $\overline{\partial}_{z'}^{*}$ the formal adjoint of the $\overline{\partial}$-operator in the $z'$ variables. Inserting \eqref{key} into the first term in the third line of \eqref{keycalc} shows that
\begin{multline}\label{keycalc2}
\int_{\mathbb{D}^q} \langle \dbar \phi\wedge \alpha, \overline{\gamma_{1}}d\zb_1\wedge\cdots\wedge d\zb_q\rangle dV(z') = 
\int_{\mathbb{D}^q} \langle \dbar H^{q-1}_{\phi}\alpha,  \overline{\gamma_{1}}d\zb_1\wedge\cdots\wedge d\zb_q\rangle dV(z') \\
= \int_{\mathbb{D}^q} \langle  H^{q-1}_{\phi}\alpha\;, \,\overline{\partial}_{z'}^{*}  
(\overline{\gamma_{1}}d\zb_1\wedge\cdots\wedge d\zb_q)\rangle dV(z')\;. 
\end{multline}
We have used here that terms in $\overline{\partial}H_{\phi}^{q-1}\alpha$ and $H_{\phi}^{q-1}\alpha$ that contain differentials $d\overline{z_{s}}$ with $s\geq (q+1)$ drop out upon taking inner products with $\overline{\gamma_{1}}d\overline{z_{1}}\wedge\cdots\wedge d\overline{z_{q}}$ or $\overline{\partial}_{z'}^{*}(\overline{\gamma_{1}}d\overline{z_{1}}\wedge\cdots\wedge d\overline{z_{q}})$, respectively, and that for $z''\in \Omega_{2}$ fixed, $\gamma$ is compactly supported in $\mathbb{D}^{q}$.

Now let $\{F_{j}\}_{j=1}^{\infty} \subset A^{2}(\Omega)$ be a bounded sequence whose sequence of  restrictions $\{f_{j}\}_{j=1}^{\infty} \subset A^{2}(\Omega_{2})$ does not admit a convergent subsequence, and set $\alpha_{j}=F_{j}d\zb_2\wedge \cdots\wedge d\zb_q\in A^2_{(0,q-1)}(\D)$ (with the convention from above when $q=1$). Then we have from \eqref{keycalc} and \eqref{keycalc2}
\begin{multline}
(c_{\chi})^q(f_j(z'')-f_k(z''))
= \int_{\mathbb{D}^q} \langle  H^{q-1}_{\phi}(\alpha_j-\alpha_k),
\overline{\partial}_{z'}^{*}(\overline{\gamma_{1}}d\zb_1\wedge\cdots\wedge d\zb_q)\rangle dV(z') \\
+ \int_{\mathbb{D}^q} \langle \dbar \phi\wedge (\alpha_{j}-\alpha_{k}), (\overline{\gamma}-\overline{\gamma_{1}})d\zb_1\wedge\cdots\wedge d\zb_q\rangle dV(z') 
\end{multline}
for any $j,k$, and thus
\begin{multline} \label{Eqn4}
|f_j(z'')-f_k(z'')|^2 
\lesssim C_{\gamma_{1}}\int_{\mathbb{D}^q} |H^{q-1}_{\phi}(\alpha_j-\alpha_k)(z',z'')|^2dV(z')\\
+\; \int_{\mathbb{D}^{q}}\left|(\alpha_{j}-\alpha_{k})(z',z'')\right|\,\left|(\overline{\gamma}-\overline{\gamma_{1}})(z',z'')\right|dV(z') \;,
\end{multline}
where $C_{\gamma_{1}}$ is a constant that depends on $\gamma_{1}$.
Integrating both sides of \eqref{Eqn4} with respect to $z''\in \D_2$ gives 
\begin{equation}\label{Eqn3}
\|f_j-f_k\|^2_{A^{2}(\D_2)}\lesssim 
C_{\gamma_{1}}\|H^{q-1}_{\phi}(\alpha_j-\alpha_k)\|^2_{L^2_{(0,q-1)}(\D)} + (\sup_{\overline{\mathbb{D}^{q}}\times\overline{\Omega_{2}}}|\gamma-\gamma_{1}|) \|\alpha_{j}-\alpha_{k}\|_{L^{2}_{(0,q-1)}(\Omega)} \;.
\end{equation} 
Assume now that $H^{q-1}_{\phi}$ is compact on $A^{2}_{(0,q-1)}(\Omega)$. Then the sequence $\{H^{q-1}_{\phi}\alpha_{j}\}_{j=1}^{\infty} \subset L^{2}_{(0,q-1)}(\Omega)$ has a subsequence, say $\{H^{q-1}_{\phi}\alpha_{j_{s}}\}_{s=1}^{\infty}$, that is convergent, and so is a Cauchy sequence. Because $\{\alpha_{j}\}_{j=1}^{\infty}$ is bounded in $L^{2}_{(0,q-1)}(\Omega)$, and we can make $(\sup_{\overline{\mathbb{D}^{q}}\times\overline{\Omega_{2}}}|\gamma-\gamma_{1}|)$ as small as we wish, \eqref{Eqn3} implies that the sequence $\{f_{j_{s}}\}_{s=1}^{\infty} \subset A^{2}(\Omega_{2})$ is  Cauchy as well, and therefore is convergent. This is a contradiction. Therefore, $H^{q-1}_{\phi}$ cannot be compact on $A^{2}_{(0,q-1)}(\Omega)$. 
\end{proof}
It is worth noting that in the last part of this proof (from \eqref{keycalc} on), the two key steps are the exploitation of nonanalyticity of $\phi$ via the introduction of $\overline{\partial}\phi$ into the mean value equation \eqref{keycalc}, and the observation \eqref{key}. The extra complications in the formulas arise from approximating $\gamma$ by a smooth function. This step is needed because 
we can only assert that $\gamma\in C(\overline{\Omega})$ from $\phi\in C^{1}(\overline{\Omega})$, yet in \eqref{keycalc2}, $\gamma$ (resp. $\gamma_{1}$) is differentiated (via $\overline{\partial}_{z'}^{*}$). These complications could be avoided by assuming $\phi\in C^{2}(\overline{\Omega})$.

\section{Proofs of Theorem \ref{ThmHankel} and Corollary \ref{Cor1}}

We start with the proof of Theorem \ref{ThmHankel}.
\begin{proof}[Proof of Theorem \ref{ThmHankel}:] 
It suffices to show that for every $\varepsilon >0$, there exists $C_{\varepsilon}$ so that we have the family of estimates
\begin{equation}\label{CompHan}
 \|H_{\phi}^{q-1}f\|^{2} \leq \varepsilon\|f\|^{2} + C_{\varepsilon}\|f\|_{-1}^{2}\;,\; f\in K_{(0,q-1)}^{2}(\Omega)\; .
 \end{equation}
Because $K^{2}_{(0,q-1)}(\Omega)$ embeds compactly into $W^{-1}_{(0,q-1)}(\Omega)$, this family of estimates will imply that $H_{\phi}^{q-1}: K^{2}_{(0,q-1)}(\Omega) \rightarrow L^{2}_{(0,q-1)}(\Omega)$ is compact (\cite{StraubeBook}, Lemma 4.3; in fact, compactness of $H_{\phi}^{q-1}$ is equivalent to this family of estimates). Note that the left hand side of \eqref{CompHan} equals
\begin{equation}\label{simpleHan}
\langle H^{q-1}_{\phi}f,H^{q-1}_{\phi}f\rangle = 
\langle \dbar^*N_q(\dbar\phi\wedge f),\dbar^*N_q(\dbar\phi\wedge f)\rangle
= \langle N_q(\dbar\phi\wedge f),\dbar\phi\wedge f\rangle \; .
\end{equation}
We will estimate the right hand side of \eqref{simpleHan}.

Denote by $A$ the union of all the $(n-1)$-dimensional analytic (then actually affine, by convexity, \cite{FuStraube98}, \cite{CuckovicSahutoglu09}, Lemma 2) varieties in the boundary. Near the boundary, the split of forms into their normal and tangential components is well defined. A detailed discussion may be found in \cite{StraubeBook}, section 2.9. The tangential component $(\overline{\partial}\phi)_{Tan}$ of $\overline{\partial}\phi$ vanishes at points of $\overline{A}$. For $\varepsilon >0$, denote by $U_{\varepsilon}$ a neighborhood of $\overline{A}$ in $\mathbb{C}^{n}$ such that $|(\overline{\partial}\phi)_{Tan}| < \varepsilon$ on $U_{\varepsilon}\cap\overline{\Omega}$, and choose a cutoff function $\chi_{1}\in C^{\infty}_{0}(U_{\varepsilon})$ with $\chi_{1} \equiv 1$ near $\overline{A}$. Then
\begin{equation}\label{easy}
 \left|\langle\chi_{1} N_q(\dbar\phi\wedge f),(\dbar\phi)_{Tan}\wedge f\rangle\right| \leq \|N_q(\dbar\phi\wedge f)\|\|\chi_{1}((\dbar\phi)_{Tan}\wedge f)\| \lesssim \varepsilon\|f\|^{2}\;.
\end{equation}
To estimate the contribution from the normal component $(\overline{\partial}\phi)_{Norm}$ of $\overline{\partial}\phi$, notice that only the normal component $\big(\chi_{1}N_{q}(\overline{\partial}\phi\wedge f)\big)_{Norm}$ will be involved (as $\big((\overline{\partial}\phi)_{Norm}\wedge f\big)$ has vanishing tangential component).
Estimate 2.91 in \cite{StraubeBook} provides the estimate
\begin{multline}\label{normest}
 \;\;\;\;\|(\chi_{1}N_{q}(\overline{\partial}\phi\wedge f))_{Norm}\|_{1} 
 \lesssim 
 \|\overline{\partial}\big(\chi_{1}N_{q}(\overline{\partial}\phi\wedge f)\big)\| 
 + \|\overline{\partial}^{*}\big(\chi_{1}N_{q}(\overline{\partial}\phi\wedge f)\big)\| \\
 \lesssim \|N_{q}(\overline{\partial}\phi\wedge f)\| + \|\overline{\partial}\phi\wedge f\|
  \lesssim \|f\| \;.\;\;\;\;
\end{multline}
Because $W^{1}(\Omega)$ imbeds compactly into $L^{2}(\Omega)$, the map $f \rightarrow \chi_{1}N_{q}(\overline{\partial}\phi\wedge f)_{Norm}$ is compact from $K^{2}_{(0,q-1)}(\Omega)$ to $L^{2}_{(0,q)}(\Omega)$. Therefore (\cite{StraubeBook}, Lemma 4.3)
\begin{equation}\label{Norm}
 \|\chi_{1}N_{q}(\overline{\partial}\phi\wedge f)_{Norm}\| \leq \varepsilon\|f\| + C_{\varepsilon}\|f\|_{-1} 
\end{equation}
(i.e. for all $\varepsilon > 0$, there exists $C_{\varepsilon}$ such that \eqref{Norm} holds). This
gives
\begin{equation}\label{easy*}
 \left|\langle\chi_{1} N_q(\dbar\phi\wedge f),(\dbar\phi)_{Norm}\wedge f\rangle\right| \lesssim
 \big(\varepsilon\|f\| + C_{\varepsilon}\|f\|_{-1}\big)\|f\| \leq 2\varepsilon\|f\|^{2} + C_{\varepsilon}\|f\|_{-1}^{2}\;.
 \end{equation}
Here, we have used the usual small constant--large constant estimate on $\|f\|_{-1}\|f\|$, and we have allowed $C_{\varepsilon}$ to change its value. It remains to estimate $\langle(1-\chi_{1}) N_q(\dbar\phi\wedge f),\dbar\phi\wedge f\rangle$.

\smallskip

In estimating this latter contribution, we use two observations. The first is that functions in $C(\overline{\Omega})$ that vanish on $\overline{A}$ are compactness multipliers for $(0,n-1)$-forms (\cite{CelikStraube09}, Proposition 1 and Theorem 3). The second observation is that, more or less, norms of $(0,q)$-forms can be estimated by norms of certain associated $(0,n-1)$-forms (for which we can then apply the compactness estimate).

We elaborate on the second observation. Suppose $u\in L^{2}_{(0,q)}(\Omega)$ is supported in a special boundary chart, with vanishing normal component, say $u=\sideset{}{'}\sum_{|J|=q, n\notin J}u_{J}\overline{\omega_{J}}$. Fix a multi-index $I$ of length $(n-1-q)$, with $n\notin I$. Then
\begin{equation}\label{n-1}
 u\wedge\overline{\omega_{I}} = \sideset{}{'}\sum_{|J|=q,n\notin J}u_{J}(\overline{\omega_{J}}\wedge\overline{\omega_{I}}) = \epsilon^{I^{c},I}_{(1,\ldots,n-1)}u_{I^{c}}(\overline{\omega_{1}}\wedge\cdots\wedge\overline{\omega_{n-1}})\;,
\end{equation}
where $I^{c}$ is the increasingly ordered multi-index of length $q$ which as a set is the complement of $I$ in $\{1,\ldots,n-1\}$, and $\epsilon^{I^{c},I}_{(1,\ldots,n-1)}$ denotes the usual Kronecker symbol. \eqref{n-1} shows that taking the wedge product with $\overline{\omega_{I}}$ singles out precisely one coefficient of $u$ (namely $u_{I^{c}}$). If we now let $I$ vary over all multi-indices of length $(n-1-q)$ (and not containing $n$), $I^{c}$ will vary over all indices of length $q$. Therefore, to estimate $\|u\|$, it suffices to estimate $\|u\wedge\overline{\omega_{I}}\|$, for all such $I$. The point of not having $n$ in $I$ is that we will want to use compactness estimate on $u\wedge\overline{\omega_{I}}$, which requires the form to be in the domain of $\overline{\partial}^{*}$. In order to apply this scheme to the form $N_{q}(\overline{\partial}\phi\wedge f)$, which appears on the right hand side of \eqref{simpleHan}, we need to localize and also take care of normal components. To this end, choose cutoff functions $\chi_{2},\ldots,\chi_{m}$ so that together with $\chi_{1}$ from above, they form a partition of unity near $b\Omega$, and so that for $2\leq s\leq m$, $\chi_{s}$ is supported in a special boundary chart. Moreover, $\chi_{2},\cdots,\chi_{m}$ can be chosen so that the supports stay close enough to $b\Omega$ that splitting forms into their tangential and normal components is well defined.
Also set $\chi_{0}:= 1-(\chi_{1}+\cdots+\chi_{m})$ on $\Omega$; then $\chi_{0}\in C^{\infty}_{0}(\Omega)$.

Fix an $s$ with $2\leq s\leq m$, and consider $\chi_{s}N_{q}(\overline{\partial}\phi\wedge f)$. Note that multiplication by $\chi_{s}$ preserves the domain of $\overline{\partial}^{*}$. The normal component of a form in $dom(\overline{\partial})\cap dom(\overline{\partial}^{*})$ is in $W^{1}_{0}(\Omega)$, and so is also in $dom(\overline{\partial})\cap dom(\overline{\partial}^{*})$ (\cite{StraubeBook}, section 2.9). Therefore, so is the tangential component. Moreover, we have as in \eqref{normest} and \eqref{Norm}
 \begin{equation}\label{normest2}
 \|(\chi_{s}N_{q}(\overline{\partial}\phi\wedge f))_{Norm}\|_{1} \lesssim \|\overline{\partial}\phi\wedge f\| \lesssim \|f\| \;,
\end{equation}
and
\begin{equation}\label{Norm2}
 \|\chi_{s}N_{q}(\overline{\partial}\phi\wedge f)_{Norm}\| \leq \varepsilon\|f\| + C_{\varepsilon}\|f\|_{-1} \;.
\end{equation}
 For economy of notation, let us denote the tangential component of $\chi_{s}N_{q}(\overline{\partial}\phi\wedge f)$ by $u_{s}$. Then $u_{s}\in dom(\overline{\partial})\cap dom(\overline{\partial}^{*})$, and the discussion above applies: we only have to estimate the forms $u_{s}\wedge\overline{\omega_{I}}$, where $I$ varies over all multi-indices of length $(n-1-q)$ that do not contain $n$. Note that for such $I$, $u_{s}\wedge\overline{\omega_{I}}\in dom(\overline{\partial})\cap dom(\overline{\partial}^{*})$ (see Lemma \ref{mainlemma} in the appendix). If $\tilde{\chi_{s}}$ is a cutoff function supported in $\overline{\Omega}\setminus \overline{A}$, it is a compactness multiplier for $(0,n-1)$-forms (\cite{CelikStraube09}, Proposition 1 and Theorem 3). If in addition $\tilde{\chi}\equiv 1$ on the support of $\chi_{s}$, then $u_{s}\wedge\overline{\omega_{I}} = \tilde{\chi}(u_{s}\wedge\overline{\omega_{I}})$. Therefore, for all $\varepsilon >0$, there exists $C_{\varepsilon}$ such that
\begin{multline}\label{comp2}
 \;\;\;\;\;\;\;\|u_{s}\wedge\overline{\omega_{I}}\| = \|\tilde{\chi}(u_{s}\wedge\overline{\omega_{I}})\| \\
 \leq \varepsilon \left(\|\overline{\partial}(u_{s}\wedge\overline{\omega_{I}})\| + \|\overline{\partial}^{*}(u_{s}\wedge\overline{\omega_{I}})\|\right) +
 C_{\varepsilon} \|u_{s}\wedge\overline{\omega_{I}}\|_{-1} \;.\;\;\;\;\;\;\;\;\;\;
\end{multline}
Estimate \eqref{starest} in Lemma \ref{mainlemma} gives
\begin {multline}\label{easy2}
 \|\overline{\partial}(u_{s}\wedge\overline{\omega_{I}})\| + \|\overline{\partial}^{*}(u_{s}\wedge\overline{\omega_{I}})\| \lesssim \|\overline{\partial}u_{s}\| + \|\overline{\partial}^{*}u_{s}\| \\
 \lesssim \|\overline{\partial}(\chi_{s}N_{q}(\overline{\partial}\phi\wedge f))\| + \|\overline{\partial}^{*}(\chi_{s}N_{q}(\overline{\partial}\phi\wedge f))\| + \|f\| \\
 \lesssim \|N_{q}(\overline{\partial}\phi\wedge f)\| + \|\overline{\partial}^{*}N_{q}(\overline{\partial}\phi\wedge f)\| + \|f\| \lesssim \|f\|\;.
\end {multline}
In the second estimate, we have used \eqref{normest2}, which implies that $\|\overline{\partial}u_{s}\|\leq \|\overline{\partial}\big(\chi_{s}N_{q}(\overline{\partial}\phi\wedge f)\big)\|+\|f\|$, as well as the analogous estimate for $\|\overline{\partial}^{*}u_{s}\|$ (since $u_{s}=\chi_{s}N_{q}(\overline{\partial}\phi\wedge f)-\big(\chi_{s}N_{q}(\overline{\partial}\phi\wedge f)\big)_{Norm})$.
We point out that estimating the term $\|\overline{\partial}(u_{s}\wedge\overline{\omega_{I}})\|$ is straightforward; it is only in estimating $\|\overline{\partial}^{*}(u_{s}\wedge\overline{\omega_{I}})\|$ that the assumption on maximal estimates is needed. For the last term in \eqref{comp2}, we observe that because the forms $\omega_{I}$ are smooth up to the boundary
\begin {equation}\label{easy3}
 \|u_{s}\wedge\overline{\omega_{I}}\|_{-1} \lesssim  \|u_{s}\|_{-1} \leq \varepsilon\|f\| + C_{\varepsilon}\|f\|_{-1} \;,
\end {equation}
for a suitable $C_{\varepsilon}$.
The second inequality follows again with \cite{StraubeBook}, Lemma 4.3, because the map $f\rightarrow u_{s}$ is continuous into $L^{2}_{(0,q)}(\Omega)$, hence compact into $W^{-1}_{(0,q)}(\Omega)$.

Combining \eqref{Norm2} through \eqref{easy3}, we find
\begin{equation}\label{combined}
 \|\chi_{s}N_{q}(\overline{\partial}\phi\wedge f)\| \leq \varepsilon \|f\| + C_{\varepsilon}\|f\|_{-1}\;,
\end{equation}
again for $C_{\varepsilon}$ suitably big. Therefore (as in \eqref{Norm},\eqref{easy*}), 
\begin{multline}\label{chi-s}
 \;\;\;\left|\langle \chi_{s}N_q(\dbar\phi\wedge f),\dbar\phi\wedge f\rangle\right| \leq 
 \|\chi_{s}N_q(\dbar\phi\wedge f)\|\|\overline{\partial}\phi\wedge f\| \\
 \lesssim \;\left(\varepsilon\|f\| + C_{\varepsilon}\|f\|_{-1}\right)\|f\| \;\;\lesssim \;\;2\varepsilon\|f\|^{2} + C_{\varepsilon}\|f\|_{-1}^{2}\;.\;\;\;\;\;\;
\end{multline}

It remains to estimate the contribution from the factor $\chi_{0}$ to the right hand side of \eqref{simpleHan}. This is a consequence of interior elliptic regularity. A short argument is as follows. Because $\chi_{0}$ vanishes on the boundary, it is a compactness multiplier, so that \cite{CelikStraube09}, Proposition 1 and Remark 2 give the same estimate as \eqref{combined}, but with $\chi_{0}$ in place of $\chi_{s}$. In turn, we obtain, as in \eqref{chi-s},
\begin{equation}\label{interior}
\left|\langle \chi_{0}N_q(\dbar\phi\wedge f),\dbar\phi\wedge f\rangle\right| 
 \lesssim \;2\varepsilon\|f\|^{2} + C_{\varepsilon}\|f\|_{-1}^{2}\;.
\end{equation}
We have used that in (32) in \cite{CelikStraube09}, it is immaterial whether the estimate is stated with $\|\cdot\|$ or with $\|\cdot\|^{2}$.

\medskip

\eqref{simpleHan} together with \eqref{easy}, \eqref{easy*}, \eqref{chi-s}, 
and \eqref{interior} establish the family of estimates in \eqref{CompHan}. 
This completes the proof of Theorem \ref{ThmHankel}.
\end{proof}

We complete this section by proving Corollary \ref{Cor1}.
\begin{proof}[Proof of Corollary \ref{Cor1}:]
We argue indirectly. Let $V$ a $q$-dimensional analytic variety in $b\Omega\setminus\overline{A}$. We may assume that $\overline{V}\cap\overline{A} = \emptyset$, and furthermore, that $V$ is smooth (otherwise, choose a small enough subset of $V$ near a regular point of $V$ in $b\Omega\setminus \overline{A}$). Choose a symbol $\phi\in C^{\infty}(\overline{\Omega})$ that vanishes identically on $\overline{A}$ and is \emph{not} holomorphic on $V$. 
Because we have maximal estimates for $(0,q)$-forms (i.e. the comparable eigenvalues condition at level $q$), and $\phi$ is (trivially) holomorphic on every $(n-1)$-dimensional variety in the boundary, Theorem \ref{ThmHankel} implies that $H^{q-1}_{\phi}: K^{2}_{(0,q-1)}(\Omega) \rightarrow L^{2}_{(0,q-1)}(\Omega)$ is compact. This contradicts Theorem \ref{ThmNonCompact}.
 \end{proof}

\section{Proof of Theorem \ref{ThmComp}}

\begin{proof}[Proof of Theorem \ref{ThmComp}:]
First note that both compactness and subellipticity of the $\overline{\partial}$-Neumann problem are known to percolate up (see for example \cite{StraubeBook}, Proposition 4.4 and the remark following its proof; the proof shows that the subelliptic gain does not decrease). Therefore, we only have to show the downward percolation in both (i) and (ii) under the assumptions in Theorem \ref{ThmComp}. We do this for (ii) first.

To prove the downward percolation in (ii), let $1\leq q\leq (n-1)$. We need to prove the estimate
\begin{equation}\label{subelliptic1}
 \|f\|_{\varepsilon}^{2} \lesssim \|\overline{\partial}f\|^{2} + \|\overline{\partial}^{*}f\|^{2}\;,\; f\in dom(\overline{\partial})\cap dom(\overline{\partial}^{*})\subset L^{2}_{(0,q)}(\Omega)\;,
\end{equation}
provided such an estimate (with the same $\varepsilon$) holds for $(0,n-1)$-forms. The argument follows section 3 closely.

Via a partition of unity, it again suffices to check this estimate for forms supported in a special boundary chart. Thus, 
\begin{equation}\label{split}
\|f\|_{\ep}\lesssim  
\|f_{Norm}\|_{\ep}+\|f_{Tan}\|_{\ep}\;,
\end{equation}
where $f_{Norm}$ and $f_{Tan}$ denote the normal and tangential components of $f$, respectively (see again \cite{StraubeBook}, section 2.9). The Sobolev-$1$ estimate for $f_{Norm}$ (\cite{StraubeBook}, Lemma 2.12) says that 
\begin{equation}\label{normalsubel}
 \|f_{Norm}\|_{\varepsilon} \lesssim \|f_{Norm}\|_{1} \lesssim \|\dbar f\|+\|\dbar^*f\| \;.
\end{equation}
As in section 3, equation \eqref{n-1}, one can see that in order to estimate the second term on the right hand side of \eqref{split}, it suffices to estimate $\|f_{Tan}\wedge\overline{\omega_{I}}\|_{\varepsilon}$ for all (increasing) multi-indices $I$ of length $(n-1-q)$ with $n\not\in I$. The form $(f_{Tan}\wedge\overline{\omega_{I}})$ is a $(0,n-1)$-form, and we can use the subelliptic estimate that is assumed at this form level. The result is
\begin{equation}\label{tansubel}
 \|f_{Tan}\wedge\overline{\omega_{I}}\|_{\ep} 
 \lesssim
  \|\dbar(f_{Tan}\wedge\overline{\omega_{I}})\|
+\|\dbar^*(f_{Tan}\wedge\overline{\omega_{I}})\| \;.
\end{equation}
Lemma \ref{mainlemma} in section \ref{SectionAppendix} says that the right 
hand side in \eqref{tansubel} is dominated by 
$\|\overline{\partial}f_{Tan}\| + \|\overline{\partial}^{*}f_{Tan}\|$. 
In turn, this sum is dominated by $\|\overline{\partial}f\| + 
\|\overline{\partial}^{*}f\|$ (in view of the second inequality in 
\eqref{normalsubel} and since $f_{Tan} = f - f_{Norm}$), as in \eqref{easy2}. 
With this, and \eqref{normalsubel} and \eqref{tansubel}, 
the estimate \eqref{subelliptic1} is established.

\smallskip

The proof for the downward percolation in (i) is analogous.
\end{proof}

\emph{Remark:}
Part (ii) of Theorem \ref{ThmComp} may be known, at least 
in principle. Namely, if the $\overline{\partial}$--Neumann problem is subelliptic 
on $(0,n-1)$--forms, then the maximum order of contact of $(n-1)$--dimensional 
complex manifolds with the boundary is finite and equals the reciprocal of the 
subelliptic gain (\cite{Catlin83}, Theorem 1 and (1.3)). By \cite{BloomGraham77}, 
the commutator type $m$ of the boundary also equals this quantity. Finite 
commutator type $m$ and a comparable eigenvalues condition at level $q$ 
then imply $1/m$ -- subellipticity for $(0,q)$--forms as well (\cite{Koenig15}, 
Theorem 1.2). Nevertheless, our proof, which is quite simple and direct, 
is still of interest. Of course, it is also needed for part (i) of 
Theorem \ref{ThmComp}.

\section{Appendix} \label{SectionAppendix}
The comparable eigenvalues conditions in Theorems \ref{ThmHankel} and \ref{ThmComp} are used only to 
see that $(u\wedge\overline{\omega}_{I}) \in dom(\overline{\partial}^{*})$ if $u\in dom(\overline{\partial})\cap dom(\overline{\partial}^{*})$, and to control $\|\overline{\partial}^{*}(u\wedge\overline{\omega}_{I})\|$, where $u$ is supported in an appropriate boundary chart (and $n\not\in I$). This requires a computation which we present in this appendix; no originality is claimed. Throughout this section $\Omega$ denotes a smooth bounded pseudoconvex domain.
\begin{lemma}\label{Lderivative}
Assume that the Levi form of $b\Omega$ satisfies a comparable eigenvalues condition at level $q$ for some $q$ with $1\leq q\leq (n-1)$. Let $u \in dom(\overline{\partial})\cap dom (\overline{\partial}^{*})\subset L^{2}_{(0,q)}(\Omega)$ be supported in a special boundary chart. Then $L_{k}u$ (computed in the sense of distributions) is actually in $L^{2}_{(0,q)}(\Omega)$, and
\begin{equation}\label{Lest}
\|L_{k}u\| \lesssim \|\overline{\partial}u\| + \|\overline{\partial}^{*}u\| \;,\; 1\leq k\leq (n-1)\;. 
\end{equation}
\end{lemma}
\begin{proof}
The comparable eigenvalues condition entails maximal estimates for $(0,q)$-forms, as in \eqref{EqnMaximalEst}. A little care is needed because \eqref{EqnMaximalEst} is an \emph{a priori} estimate: it is assumed that the form is smooth up to the boundary. However, forms smooth up to the boundary are dense in the graph norm of $\overline{\partial}\oplus\overline{\partial}^{*}$ (\cite{StraubeBook}, Proposition 2.3), and the proof shows that the approximation can be done with forms supported in the same boundary chart. If the approximating sequence is $\{u_{j}\}_{j=1}^{\infty}$, then \eqref{Lderivative} shows that $\{L_{k}u_{j}\}_{j=1}^{\infty}$ is Cauchy, hence converges, in $L^{2}_{(0,q)}(\Omega)$. But it also converges to $L_{k}u$ in the  sense of distributions. Therefore, $L_{k}u\in L^{2}_{(0,q)}(\Omega)$, and \eqref{Lest} holds.
\end{proof}
What is needed in the proofs of Theorems \ref{ThmHankel} and \ref{ThmComp} is contained in the next Lemma.
\begin{lemma}\label{mainlemma}
Assume that the Levi form of $b\Omega$ satisfies a comparable eigenvalues condition at level $q$ for some $q$ with $1\leq q\leq (n-1)$. Let $u \in dom(\overline{\partial})\cap dom (\overline{\partial}^{*})\subset L^{2}_{(0,q)}(\Omega)$ be supported in a special boundary chart. Then $u\wedge\overline{\omega}^{I} \in dom(\overline{\partial})\cap dom (\overline{\partial}^{*})$, and
\begin{equation}\label{starest}
 \|\dbar^*(u\wedge \ob_I)\| +\|\dbar(u\wedge \ob_I)\|
\lesssim  \|\dbar u\|+\|\dbar^* u\|\;, 
\end{equation}
for any multi--index $I$ of length less than or equal to $(n-1-q)$ that 
does not contain $n$.
\end{lemma}
Note that the issue is only with the first term on the left hand side of \eqref{starest}, the estimate for $\|\overline{\partial}(u\wedge\overline{\omega_{I}})\|$ is trivial, since $\|u\|\lesssim \|\overline{\partial}u\| + \|\overline{\partial}^{*}u\|$.
\begin{proof}[Proof of Lemma \ref{mainlemma}:]
We only need to prove the case $|I|=1$, i.e. $\overline{\omega_{I}}=\overline{\omega_{k}}$ for some $k$ with $1\leq k\leq (n-1)$; repeated application of this case then yields estimate \eqref{starest} for general $I$.

 \smallskip
 
First recall that the adjoint of wedging with $\overline{\omega_{k}}$ is essentially an interior product with $\overline{L_{k}}$. Indeed, if $v=\sideset{}{'}\sum_{|J|=q+1}v_{J}\overline{\omega_{J}}\in L^{2}_{(0,q+1)}(\Omega)$, then
 \begin{equation}\label{wedge}
 \langle u\wedge\overline{\omega_{k}}, v\rangle = \langle\sideset{}{'}\sum_{|K|=q}u_{K}(\overline{\omega_{K}}\wedge\overline{\omega_{k}}), \sideset{}{'}\sum_{|J|=q+1}v_{J}\overline{\omega_{J}}\;\rangle  
 = \sideset{}{'}\sum_{|K|=q}\langle u_{K}, v_{Kk}\rangle = \langle u, v_{k}\rangle \;,
 \end{equation}
where $v_{k}$ is the $(0,q)$-form $v_{k} = \sideset{}{'}\sum_{|K|=q}v_{Kk}\overline{\omega_{K}}$.
We have slightly abused notation: the various appearances of $\langle\cdot, \cdot\rangle$ denote the inner product between $(0,q+1)$-forms, functions, and $(0,q)$-forms, respectively. For $v=\overline{\partial}g$, \eqref{wedge} gives
 \begin{equation}\label{wedgeadjoint}
 \langle u\wedge\overline{\omega_{k}}, \;\overline{\partial}g\rangle = \langle u, (\overline{\partial}g)_{k}\rangle\;.
 \end{equation}

 \smallskip
 
Next, we compare $(\overline{\partial}g)_{k}$ to $\overline{\partial}(g_{k})$.\footnote{This amounts to a $\overline{\partial}$ version of the Cartan formula $i(X)d+di(X)=Lie_{X}$ (see for example \cite{Morita01}, Theorem 2.11). The difference in sign in \eqref{compare} below (i.e. $\overline{\partial}(g_{k})$ instead of $-\overline{\partial}(g_{k})$) results from the definition of $(\cdot)_{k}$, which corresponds to inserting $\overline{L_{k}}$ into the last slot rather than the first. This affects $(\overline{\partial}g)_{k}$ by a factor $(-1)^{q}$, and $\overline{\partial}(g_{k})$ by a factor of $(-1)^{q-1}$.} Assume for the moment that $g\in dom(\overline{\partial}^{*})$ is smooth up to the boundary, say 
$g=\sumprime_{|J|=q} g_J\ob_J$ in the local boundary frame (we do not assume that $g$ is supported in that chart). Then
$g_k=\sumprime_{|K|=q-1} g_{Kk}\ob_K\;$. Also
\begin{multline}\label{dbarsub}
(\dbar g)_k=\sumprime_{|J|=q} \sum_{m=1}^n\left( (\overline{L}_m g_J)
\ob_m\wedge \ob_J\right)_k + O(|g|)\\
=\sumprime_{|K|=q-1} \sum_{m=1}^n(\overline{L}_m g_{Kk})(\ob_m\wedge \ob_{Kk})_{k}
+(-1)^q\sumprime_{|J|=q,k\not\in J} (\overline{L}_k g_J) \ob_J
+O(|g|)\;.
\end{multline}
The first sum in the second line results from those $J$ with $k\in J$. Note that
$(\overline{\omega_{m}}\wedge\overline{\omega_{Kk}})_{k} = \overline{\omega_{m}}\wedge\overline{\omega_{K}}$ when $m\not= k$\footnote{Denote by $S$ the increasing reordering of $(m, k_{1}, \cdots, k_{q-1})$. 
Then, when $m\not= k$, $(\overline{\omega_{m}}\wedge\overline{\omega_{Kk}})_{k}=(\epsilon^{S}_{(m,k_{1},\cdots,k_{q-1})}
\overline{\omega_{Sk}}\,)_{k}=\epsilon^{S}_{(m,k_{1},\cdots,k_{q-1})}\overline{\omega_{S}}=\overline{\omega_{m}}\wedge\overline{\omega_{K}}\,$.}. When $m=k$, the term vanishes. Therefore, modulo terms that are $O(|g|)$, this sum contains all the terms of $\overline{\partial}(g_{k})$, except $\sideset{}{'}\sum_{|K|=q-1}\overline{L_{k}}g_{Kk}(\overline{\omega_{k}}\wedge\overline{\omega_{K}})$. In other words,
\begin{equation}\label{dbarsub2}
 \sumprime_{|K|=q-1} \sum_{m=1}^n(\overline{L}_m g_{Kk})(\ob_m\wedge \ob_{Kk})_{k} = \overline{\partial}\big(g_{k}\big) + (-1)^{q} \sideset{}{'}\sum_{|K|=q-1}\overline{L_{k}}g_{Kk}(\overline{\omega_{K}}\wedge\overline{\omega_{k}}) + O(|g|)\; ,
\end{equation}
where we have used that $\overline{\omega_{k}}\wedge\overline{\omega_{K}} = (-1)^{q-1}\overline{\omega_{K}}\wedge\overline{\omega_{k}}$. Taking into account that the sum on the right hand side corresponds to the last sum in \eqref{dbarsub}, but for those multi--indices $J$ that contain $k$, \eqref{dbarsub} and \eqref{dbarsub2} combine to give the comparison we seek, namely
\begin{equation}\label{compare}
 (\overline{\partial}g)_{k} = \overline{\partial}(g_{k}) + (-1)^{q}\sumprime_{|J|=q} (\overline{L}_k g_J) \ob_J + O(|g|)\; .
\end{equation}

Let now $u\in dom(\overline{\partial}^{*})$ also be smooth up to the boundary. In view of \eqref{wedgeadjoint} and \eqref{compare}, we have
\begin{multline}\label{Eqn10}
\langle u\wedge\ob_k,\dbar g \rangle = \langle u, (\overline{\partial}g)_{k}\rangle \\ 
= \big\langle \;u, \;\dbar (g_k) 
+ (-1)^q\sumprime_{|J|=q} (\overline{L}_k g_J) \ob_J \;
 \big\rangle+ O(\|u\|\|g\|)\\
=\langle \dbar^*u,g_k\rangle 
+(-1)^{q+1}\sumprime_{|J|=q} \langle \; L_k u,g_J \overline{\omega_{J}}\;\rangle 
+ O(\|u\|\|g\|)\;.
\end{multline}
In the last step, we have used that $1\leq k\leq (n-1)$ to integrate $\overline{L_{k}}$ by parts without boundary term.
Approximating $g$ in the graph norm of $\overline{\partial}$ (\cite{StraubeBook}, Proposition 2.3) shows that equality of the left hand side with the right hand side in \eqref{Eqn10} remains valid if $g$ is only assumed in $dom(\overline{\partial})$. This in turn shows 
that $(u\wedge\overline{\omega_{k}})\in dom(\overline{\partial}^*)$ and that
\begin{equation}\label{Eqn11}
\|\dbar^*(u\wedge \ob_k)\| 
\lesssim \|\dbar^* u\|+\|L_ku\|+\|u\|\;.
\end{equation}
This estimate is for $u$ smooth up to the boundary. But if we now approximate a nonsmooth $u$ in the graph norm of $\overline{\partial}\oplus\overline{\partial}^{*}$ by forms smooth up to the boundary (and supported in the same boundary chart, as above) and invoke Lemma \ref{Lderivative}, we find that $(u\wedge\overline{\omega_{k}})\in dom(\overline{\partial}^{*})$ and the estimate
\eqref{Eqn11} still holds. By \eqref{Lest} and the obvious estimate for $\|u\|$, we thus obtain
\begin{equation}\label{final1}
\|\overline{\partial}^{*}(u\wedge\omega_{k})\| \lesssim \|\overline{\partial}u\| \\ + \|\overline{\partial}^{*}u\| \;.
\end{equation}
Finally, as noted above, we also have 
\begin{equation}\label{final2}
\|\dbar (u\wedge \ob_k)\| \lesssim \|\dbar u\|+\|u\| \lesssim \|\dbar u\|+\|\dbar^* u\| \;.
\end{equation}
Estimates \eqref{final1} and \eqref{final2} give \eqref{starest} when $|I|=1$. The general case now follows inductively.
\end{proof}

\smallskip

\emph{Acknowledgement:} The authors are grateful to the referee and to the editor for various helpful comments. In particular, the referee pointed out the remark at the end of Section 4.

\providecommand{\bysame}{\leavevmode\hbox to3em{\hrulefill}\thinspace}

\end{document}